\documentclass[11pt,a4paper]{article}
\usepackage{ucs}
\usepackage[utf8x]{inputenc}
\usepackage{amsmath,amssymb,,amscd, amsfonts,amsthm,latexsym} 
 \usepackage[all,cmtip]{xy}
 \usepackage{relsize}
\title{ 
       Lifting Problem on Automorphism Groups of Cyclic Curves
       }
\author{
        Tovondrainy Christalin Razafindramahatsiaro \\
        \\        
        African Institute of Mathematical Science\\
        6 Melrose road Muizenberg, 7945 Cape Town South Africa         
}
\date{\today}

\newcommand{\F}{\mathbb{F}}
\newcommand{\E}{\mathbb{E}}
\newcommand{\K}{\mathbb{K}}

\newcommand{\Z}{\mathbb{Z}}

\newcommand{\Aut}{\mathrm{Aut}}
\newtheorem{thm}{Theorem}[section]
\newtheorem{lem}[thm]{Lemma}
\newtheorem{exa}[thm]{Example}
\newtheorem{pro}[thm]{Proposition}
\newtheorem{defn}[thm]{Definition}

\newtheorem{rem}[thm]{Remark}
\newtheorem{conj}[thm]{Conjecture}

\begin{document}
\maketitle
\begin{abstract}
Let $X$ be a smooth projective hyperelliptic curve over an alge-
braically closed field $k$ of prime characteristic $p.$ The aim of this note
is to find necessary and sufficient conditions on the automorphism
group of the curve $X$ to be lifted to characteristic zero. The results
will be generalised for a certain family of curves that we call cyclic curves.

\end{abstract}

\section{Introduction}

Let $k$ be an algebraically closed field of prime characteristic $p.$ Given a smooth projective curve $X$ over $k,$ consider a lifting $(X_0/k_0, v)$ of $X/k$ to characteristic $0,$ with the following properties:
\begin{itemize}
\item $k_0$ is the algebraically closed field of the fraction field of $W(k),$ the ring of Witt vectors over $k;$
\item  $v$ is a valuation such that $k_0 v=k;$ 
\item The curves $X/k$ and $X_0/k_0$ have the same genus.
\end{itemize}
Recall that such a lift always exists and called a good lifting of $X/k.$ 

\

 According to the following proposition:
\begin{pro}[\cite{talin}]
Let $(\F{|}k,v)$ be a valued function field in one variable where $v$\footnote{Note that a good reduction is always invariant under action of any automorphism group of $\F{|}k.$} is assumed to be invariant under action of the automorphism group $\Aut(F{|}k).$ Then there is a natural injective homomorphism
$$\phi: \Aut(F{|}k)\hookrightarrow\Aut(\F v{|} k v)$$ where $\F v{|} k v$ denotes the residue function field with respect to the valuation $v.$
\end{pro} 
Together with the two equivalent categories: 
\begin{itemize}
\item Smooth projective curves over k, and dominant morphisms;
\item Function fields of one variable over k, and k-homomorphisms.
\end{itemize}
there is a natural injective homomorphism
$$\phi: \ G_0{:=}\mathrm{Aut}_{k_0}(X_0)\hookrightarrow G{:=}\mathrm{Aut}_k(X).$$ Moreover, if $\phi$ is surjective, we say that the automorphism group $G$ \texttt{is liftable} to characteristic $0.$ So, one can ask: Under which conditions is $\phi$ surjective? The aim of this note is to answer this question.

\

 We have a partial answer when the order of the group $G$ and $p,$ the characteristic of $k,$ are relatively prime. Indeed, in \cite{grot} Expos\'e XIII \S 2, Grothendieck proved that if $p$ does not divide the order of $G,$ then $G$ could always be lifted to characteristic $0.$ However, in the case when $G$ is divisible by $p,$ the problem is not completely solved.

 Here, we restrict to the case of cyclic curves (Definition \ref{def1}) over prime characteristic field. The main reason is the fact that we know all of the groups which can occur as an automorphism group of a cyclic curve in any characteristic which is not equal to $2$ (See \cite{shaska} and \cite{sanj}). Therefore, a priori, a quite elementary covering theory, some group theory and representation of finite subgroups of $\mathrm{PGL}_2(k)$ would suffice. We will show that:

 \
 
 \textit{Let $X/k$ be a smooth projective irreducible hyperelliptic curve. Denote by $G$ its full automorphism group and suppose that $p\neq2.$ Assume that the order of $G$ is divisible by $p.$ Then, the automorphism group $G$ is liftable to characteristic $0$ if and only if, up to isomorphism, $G$ is one of the following groups:
\begin{enumerate}
\item $G=\mathbb{Z}/2p\mathbb{Z};$ 
\item $G=D_{2p};$
\item $G=\mathbb{Z}/2\mathbb{Z}\times\mathcal{A}_5$ or $\mathrm{SL}_2(5)$ if $p=5;$ 
\item \label{tg} $G=\mathbb{Z}/2\mathbb{Z}\times\mathcal{A}_4,  \mathbb{Z}/2\mathbb{Z}\times\mathcal{S}_4, \mathbb{Z}/2\mathbb{Z}\times\mathcal{A}_5, \mathrm{SL}_2(3),$ or $\mathrm{GL}_2(3)$ in the case when $p=3.$ 
\end{enumerate}}
\textit{In particular, the group $G$ is liftable to characteristic $0$ if and only if $G$ is an Oort-group for $k.$ We also observe that the order of $G$ is divisible by $p,$ but not by $p^2.$ Furthermore, we have $p\leq2g+1,$ where $g$ denotes the genus of the curve $X/k.$}
 
\
 
It is important to point out that this problem is related to \texttt{Oort groups} and \texttt{Lifting problems} (Definition \ref{oort}) that we can see in \cite{oortgroup}. By definition, an Oort group for $k$ is clearly liftable to characterisitic $0.$ But, the converse is not always true. So, a priori, for an automorphism group of a smooth projective curve over $k,$ the condition of being an Oort group is stronger than being liftable to characteristic $0.$ Our results in this present paper confirm some predictions of Chinburg, Guralnick and Harbater, in \cite{oortgroup}.

To begin with, let us first recall some results on Oort groups for $k.$ And also, some preliminary results that we will need to prove our results in the last section.

\textit{Throughout this note, $k$ denotes an algebraically closed field of prime characteristic $p.$}

\section{Preliminary Results}

\begin{defn}\label{oort}
A finite group $G$ is called an \texttt{Oort group} for $k$ if every faithful action of $G$ on a smooth connected projective curve over $k$ lifts to characteristic $0.$
\end{defn}
Oort conjectured that we can lift any cyclic group. It turns out that this conjecture is true according to the works of Pop (See \cite{pop}), mainly using deformation theory and a special case of a result by Obus and Wewers that we can see in \cite{ob}. Moreover, Chinburg, Guralnick and Harbater, in \cite{oortgroup}, proved that if $G$ is liftable to characteristic zero, then any cyclic-by-p subgroup (extensions of a prime-to-$p$ cyclic group by a $p$-group) of $G$ must be either cyclic or dihedral of the form $D_{p^n}$, with the exception of $\mathcal{A}_4$ in characteristic $2.$ They also predict that the converse is true. But as far as we know, no one has given a proof that the dihedral group $D_{p^n}$ for $n>1$ can be lifted to characteristic zero. This conjecture is called the \textit{Strong Oort conjecture}.

\

The following lemma is a summary of what we will need about Oort groups:
\begin{lem}[\cite{oortgroup}]\label{rem2}

\

\begin{itemize}

\item A $p$-regular group (its order is not divisible by $p$), a finite cyclic group, the dihedral group $D_p$ and the Klein four-group $V_4$ (in the case $p=2$) are Oort groups for $k$;
\item The quaternion group $Q_8$ (in the case $p=2$), the group $(\mathbb{Z}/p\mathbb{Z})^n$ where $n\geq2$ (resp. $>2$) if $p\neq2$ (resp. $p=2$) are not Oort groups;
\item Let $G$ be a finite group. Then $G$ is an Oort group if and only if every cyclic-by-$p$ subgroup $H\subset G$ is an Oort group;
\item  If a cyclic-by-$p$ group is an Oort group for $k,$ then it must be cyclic or dihedral of the form $D_{p^n}$ for some integer $n$. 
\end{itemize}
\end{lem}

\begin{defn}\label{def1}
We say that a function field $\F{|}k$ is \texttt{cyclic} (or \texttt{superelliptic}) if the following condition is satisfied: There exists a transcendental element $x$ such that the rational function field $k(x)$ is invariant under the action of the full automorphism group $G$ of $\F {|}k $ and the subgroup $N{=}\mathrm{Aut}\left( \F {|}k (x)\right)$ is cyclic, Galois and a normal subgroup of $G.$ 

Here, the base field $k$ is an algebraically closed field of characteristic $p\geq0.$ The smooth projective curve $X$ over $k$ associated to $\F{|}k$ is called \texttt{cyclic curve.}
\end{defn}

\

\begin{pro}\label{pro}
Let $(\F{|}\K,v)$ be a valued function field where the valuation $v$ is assumed to be invariant under action of $\Aut(\F{|}k).$ Denote by $\E$ the fixed field of the full automorphism group $G$ of $\F{|}\K.$ Then, for any subfield $\mathbb{L} v$ between $\F v$ and $\E v$, there exists a unique subfield $\mathbb{L}$ between $\F$ and $\E$ such that $\mathbb{L} v$ is the exact reduction of $\mathbb{L}$ by the valuation $v.$ 

Note that, the subfield $\mathbb{L}$ and $\mathbb{L}v$ have the same genus in the case when $v$ is a good reduction.
\end{pro}
\begin{proof}
We know that $G\simeq\mathrm{Aut}(\F v{|}\E v)$ and $\F v{|}\E v$ is Galois. The result follows by the Galois correspondence. Indeed, the extension $\F v{|}\mathbb{L} v$ is also Galois and denote by $Hv$ its Galois group. By the Galois correspondence theorem, if $n$ is the number of subgroups, $H_i$ ($1\leq i\leq n$), of $G$ which have the same order as $Hv,$ then there exists exactly $n$ extensions, $\F v{|}\mathbb{L}_i v$ ($1\leq i\leq n$), such that $H_iv{=}\mathrm{Aut}\left( \F v{|}\mathbb{L}_i v\right) $ for each $1\leq i\leq n.$ The group $H_iv$ denotes the reduction of $H_i$  by the natural isomorphism between $G$ and $\mathrm{Aut}(\F v{|}\E v)$ for each $1\leq i\leq n.$ Therefore, there must be a unique subgroup $H$ (one of the subgroups $H_i$) of $G$ such that $H\simeq Hv$ (via the natural isomorphism between $G$ and $\mathrm{Aut}(\F v{|}\E v)$) and  $\F^{H}v=\mathbb{L}v.$ Thus, $\mathbb{L}=\F^{H}$.
\end{proof}

\

Now, let $\F{|}k$ be a cyclic function field. Denote by $X$ the smooth cyclic curve over $k$ associated to $\F{|}k.$ Now, let $X_0$ be a good lifting of $X$ to characteristic $0$ over an algebraically closed field $k_0.$ Then, there exists a valued function field $\left(\F_0{|}k_0,v\right)$ such that $v$ is  a good reduction, $\F=\F_0v$ and $k=k_0v.$ Suppose that the full automorphism group $G$ of the curve $X/k$ is liftable to characteristic zero and assume that $G\simeq G_0{=}\mathrm{Aut}(\F_0{|}k_0).$ Using Proposition \ref{proo}, there exists a residually transcendental element $x$ such that the extensions $\F_0{|}k_0(x)$ and $\F{|}k(\overline{x})$ are Galois. Hence, $\F_0{|}k_0$ is also a cyclic function field
and $N{=}\mathrm{Aut}\left( \F_0{|}k_0(x)\right)\simeq\mathrm{Aut}(\F{|}k\left( \overline{x})\right).$ Note also that the curve $X_0/k_0$ is cyclic. Furthermore, the groups $G/N$ and $G_0/N$ are isomorphic and respectively embedded in $\mathrm{PGL}_2(k)$ and $\mathrm{PGL}_2(k_0).$ Therefore, it is important to know which kind of groups could be finite subgroups of both $\mathrm{PGL}_2(k)$ and $\mathrm{PGL}_2(k_0).$

\

\begin{lem}\label{biglem}
Let $H$ be a finite subgroup of $\mathrm{PGL}_2\left(k\right).$ The group $H$ can be embedded in $\mathrm{PGL}_2(k_0)$ if and only if one of the following statements holds:
\begin{itemize}
\item The prime characteristic $p$ does not divide ${\mid}H{\mid},$ the order of $H;$
\item If $p$ divides ${\mid}H{\mid},$ then up to isomorphism, $H$ is one of the following groups:
\begin{itemize}
\item $H=\mathbb{Z}/ p\mathbb{Z};$ 
\item $H=D_{p}$ if $p\neq2;$
\item $H=\mathcal{A}_4$ or a dihedral group $D_n$ where $n$ is a positive odd integer if $p=2;$
\item$H=\mathcal{A}_5$ in the case when $p\leq5;$
\item$H=\mathcal{A}_4$ or $\mathcal{S}_4$ when $p=3.$
\end{itemize}

\end{itemize} 
\end{lem}
\begin{proof}
In order to prove the lemma, we shall recall the following theorem from \cite{fini} Theorem B and Theorem C:
\begin{pro}\label{lem1}
Let $\K$ be an algebraically closed field of characteristic $q.$ Let $H$ be a finite subgroup of $\mathrm{PGL}_2(\K).$ Then:
\begin{itemize}
\item If $q=0$ or  $q>0$ and $q\nmid {|}{H}{|},$ the group $H$ is isomorphic to a cyclic group, a dihedral group, $\mathcal{A}_4, S_4$ or $\mathcal{A}_5;$
\item If $q>0$ and $q$ divides the order of $H,$ then $H$ is isomorphic to one of the following groups: $\mathrm{PGL}_2(\F_{q^n})$,$ \mathrm{PSL}_2(\F_{q^n})$ for some integer $n$ or to a $q$-semi-elementary subgroup. Note that the conjugacy classes of $q$-semi-elementary subgroups of $\mathrm{PGL}_2\left(\K\right) $ of order $p^mn$($n\in\mathbb{N}\setminus q\mathbb{N}$ and $m\in\mathbb{N}^{\star}$) are parameterized by the set of homothety classes of rank-$m$ subgroups $\Gamma$ satisfying $\F_{q^e}\subset\Gamma\subset \K$ via the map 
$$\Gamma\mapsto\left( \begin{array}{cc}
1 & \Gamma \\ 
 & 1
\end{array}\right)\rtimes\left( \begin{array}{cc}
\mu_n\left(\K\right) &  \\ 
 & 1
\end{array}\right)$$
where $e$ is the order of $q$ in $(\mathbb{Z}/n\mathbb{Z})^{\times}$ and $\mu_n\left(\K\right)$ is the group of primitive $n$-th roots of unity in $\K.$
\end{itemize}
With the following exceptional possibilities:
\begin{itemize}
\item Suppose that $q=3,$ then $H$ could also be isomorphic to $\mathcal{A}_5;$
\item If $q=2,$ $H$ is isomorphic to a dihedral group $D_n$ where  $n$ is an odd positive integer.
\end{itemize}

\end{pro}

\

Let us now prove our lemma. If $p\nmid{|}{H}{|},$ then by the first statement of Proposition \ref{lem1} , the subgroup $H$ can be embedded in $\mathrm{PGL}_2(k_0).$ Therefore for the rest of the proof of Lemma \ref{biglem}, we may assume that $p$ divides the order of $H.$ 

 Let us assume first that $p{>}5.$ Using Proposition \ref{lem1}, since $p$ divides ${|}{H}{|},$ the group $H$ is not isomorphic to one of the groups $\mathcal{A}_4, S_4$ or $\mathcal{A}_5.$ Therefore, $H$ is either cyclic or dihedral. However, the groups $\mathrm{PGL}_2(\F_{p^n})$ and $\mathrm{PSL}_2(\F_{p^n})$ for some integer $n$ are not cyclic nor dihedral. Hence, $H$ is isomorphic to a $p$-elementary group of order $p^mn$($n\in\mathbb{N}\setminus q\mathbb{N}$ and $m\in\mathbb{N}^{\star}$) parameterized by a set of homothety classes of a rank-$m$ subgroup $\Gamma$ satisfying $\F_{p^e}\subset\Gamma\subset k$ via the map 
$$\Gamma\mapsto\left( \begin{array}{cc}
1 & \Gamma \\ 
 & 1
\end{array}\right)\rtimes\left( \begin{array}{cc}
\mu_n(k) &  \\ 
 & 1
\end{array}\right)$$
where $e$ is the order of $p$ in $(\mathbb{Z}/n\mathbb{Z})^{\times}$ and $\mu_n(k)$ is the group of primitive $n$-th root of unity in $k.$ Furthermore, by the definition of $p$-elementary groups, $H$ is cyclic if $m=1$ or dihedral if $m=1$ and $n=2.$ Thus, $H\simeq\mathbb{Z}/pn\mathbb{Z}$ or $D_p$ where $n$ is a non-negative integer prime to $p.$ However if we assume that $n>1.$ This would suggest that there is a cyclic group of order $pn$ generated by the two matrices
$$\left( \begin{array}{cc}
1 & 1 \\ 
0 & 1
\end{array}\right) , \left( \begin{array}{cc}
\zeta & 0 \\ 
0 & 1
\end{array} \right)$$ where $\zeta$ is a primitive $n$-root of unity. But, those two matrices do not commute with each other. Hence, $n$ must be equal to $1.$ 

In the case when $p\leq5,$ if $H$ is a $p$-elementary group, then $H\simeq \mathbb{Z}/p\mathbb{Z}$ or $D_{p}$ with the exceptional case $\mathcal{A}_4\simeq\left(\mathbb{Z}/2\mathbb{Z}\right)^2\rtimes\mathbb{Z}/3\mathbb{Z}$ in characteristic $p=2.$

If $p=5,$ since $\mathcal{A}_5\simeq\mathrm{PSL}_2(\F_5),$ then $H$ could be isomorphic to $\mathcal{A}_5.$ 

Now if $p=3,$ since $\mathrm{PGL}_2(\F_3)\simeq\mathcal{S}_4$ and $\mathrm{PSL}_2(\F_3)\simeq\mathcal{A}_4,  H$ could be isomorphic to $\mathcal{A}_4,\mathcal{S}_4$ and $\mathcal{A}_5$ according to the third statement of Proposition \ref{lem1}. 

Finally, if $p=2,$ according to the last statement of Proposition \ref{lem1}, $H$ could be a dihedral group and $H\simeq D_n$ where $n$ is an odd integer. We know also that $\mathrm{PGL}(2,4)$ is isomorphic to $\mathcal{A}_5,$ thus, the group $H=\mathrm{PGL}(2,4)$ can be embedded in $\mathrm{PGL}_2(k_0).$

\end{proof}

\section{Lifting Automorphism Group of  Cyclic Curves}

As we have already mentioned, according to Grothendieck in \cite{grot}, if the order of the automorphism group $G$ of a smooth curve (of genus $g\geq2$) is not divisible by $p,$ then $G$ is liftable to characteristic zero. Therefore, for the rest of this chapter, the order of the automorphism group of any curve over $k,$ unless otherwise specified, is assumed to be divisible by $p.$
 
\
 
 A direct corollary of Lemma \ref{biglem} is the following:
  \begin{thm}\label{thm}
Let $(\F_0{|}k_0,v)$ be a valued cyclic function field where the base field $k_0$ is of characteristic $0.$ Suppose that the valuation $v$ is invariant under the action of the group $G_0=\mathrm{Aut}(\F_0{|}k_0)$\footnote{Note if $\F_0{|}k_0$ has good reduction at $v$ then it is invariant under $G_0$ as good reduction is unique for $g\geq1.$}. Denote by $N$ the normal subgroup of $G_0$ as we defined in \ref{def1}.

 If $G_0$ is isomorphic to the full automorphism group $G$ of the residue function field $\F {|}k ,$ then $G/N$ is isomorphic to one of the groups: 
 \begin{itemize}
\item $\mathbb{Z}/ p\mathbb{Z};$ 
\item $D_{p}$ if $p\neq2;$
\item $\mathcal{A}_4$ or a dihedral group $D_n$ where $n$ is a positive odd integer if $p=2;$
\item$\mathcal{A}_5$ in the case when $p\leq5;$
\item$\mathcal{A}_4$ or $\mathcal{S}_4$ when $p=3.$
\end{itemize} 
with the following additional possibilities: 
 \begin{itemize} 
\item $\mathbb{Z}/m\mathbb{Z}$ or a dihedral $D_m$ if  $p$ divides ${|}N{|},$ the order of the group $N;$  

\end{itemize}

The integers $m$ is prime to $p.$
\end{thm}
\begin{proof}
 Suppose we have $G_0\simeq G.$ By hypothesis, there exists a transcendental element $x$ of $\F_0$ such that $k(x)$ is invariant under $G$ and $N=\mathrm{Aut}(\F{|}k)$ is Galois. Therefore, via the natural isomorphism between $G_0$ and $G$ and the fact that $N$ is a finite Galois group, the element $x$ is residually transcendental and the groups $G_0/N$ and $G/N$ are respectively finite subgroups of $\mathrm{PGL}_2(k_0)$ and $\mathrm{PGL}_2(k).$ Moreover, the isomorphism between $G_0$ and $G$ induces an isomorphism from $G_0/N$ to $G/N.$ Note that the group $G/N$ is a finite subgroup of $\mathrm{PGL}_2(k).$ If we assume that $p$ does not divide the order of $N,$ then $p$ must divide ${|}G/N{|}$, since $p$ divides the group $G$ by hypothesis. So, in this case, using Lemma \ref{biglem}, $G_0/N$ must be isomorphic to one of the groups:
 $$\mathbb{Z}/p\mathbb{Z}, D_p,\mathcal{A}_5,\mathcal{S}_4,\mathcal{A}_4,$$ with the exceptional group $D_m,$ where $m$ is an odd integer in characteristic $2.$
 
 Now, if $p$ divides ${|}N{|},$ the order of the group $G/N$ could be prime to $p.$ Hence, if this is the case, according to Lemma \ref{biglem}, the group $G/N$ could be isomorphic to one of the following groups:
 $$\mathbb{Z}/n\mathbb{Z}, D_n,\mathcal{A}_5,\mathcal{S}_4,\mathcal{A}_4$$ where the integer $n$ is prime to $p.$ The results follow immediately using the list of possible finite subgroups of rational function fields that can be lifted to characteristic $0$ in Lemma \ref{biglem}.
\end{proof}

\

Let us make some observations in this context:

\

Let $X/k$ be a smooth cyclic curve with automorphism group $G$ that is liftable to characteristic $0.$ Denote by $X_0/k_0$ its good lifting to characteristic $0.$ Suppose that the prime characteristic of $k$ is odd. Let $\F{|}k$ and $\F_0{|}k_0$ be the function fields which correspond to $X/k$ and $X_0/k_0$ respectively. Since the characteristic of $k$ is $p\neq2,$ we may assume that the function fields $\F_0{|}k_0$ and $\F{|}k$ are defined respectively by the equations:
$$y_0^2=P(x_0)$$
and $$y^2=\overline{P}(x)$$ with $P(x_0)\in \mathcal{O}_{k_0}\left[ x_0\right]$ where $\mathcal{O}_{k_0}$ is the valuation ring of $k_0$ which corresponds to the restriction of $v,$ the good reduction of $\F_0{|}k_0,$ to $k_0.$ The polynomial $P$ is the reduction of the polynomial $P_0$ under the Gauss valuation $v_{x_0},$ prolongation of $v$ to $k_0(x_0).$ The transcendental elements $x$ and $y$ in $\F$ are, respectively, the reductions of the transcendental elements $x_0$ and $y_0$ of $\F_0.$ 

\

Now, denote by $G_0$ and $N$ the automorphism group of $\F_0{|}k_0$\footnote{Note that the function field $F_0{|}k_0$ has the same automorphism group as the cyclic curve $X_0/k_0.$} and the normal subgroup of $G_0$ such that the quotient space $X_0/N$ has genus $0.$ In the proof of Theorem \ref{thm}, we know that there is an injective homomorphism 
$$\iota: G_0/N\hookrightarrow G/N.$$
The groups $G_0/N$ and $G/N$ are respectively subgroups of $\mathrm{Aut}(k_0(x_0){|}k_0)\simeq\mathrm{PGL}_2(k_0)$ and $\mathrm{Aut}(k(x){|}k)\simeq\mathrm{PGL}_2(k).$ The restriction of the good reduction on $\F_0$ to $k_0(x_0)$ is the Gauss valuation $v_{x_0}.$ So, we remark that:
\begin{rem}
The injective homomorphism $\iota$ is defined as follows:
\begin{align*}
\iota: & G_0/N\hookrightarrow G/N\\
&\left( \begin{array}{cc}
a & b \\ 
c & d
\end{array}\right) \mapsto \left( \begin{array}{cc}
\overline{a} & \overline{b} \\ 
\overline{c} & \overline{d}
\end{array}\right) 
\end{align*} 
where $a,b,c$ and $d$ belong to $k_0.$ For any element $u$ in $k_0,$ we denote its reduction under the valuation $v$ by $\overline{u}.$
\end{rem}

\

Let us illustrate this with an example:
\begin{exa}
 Suppose that $G_0/N=\Z/p\Z.$ Denote by $\sigma$ the generator of the group $G_0/N.$ The generator $\sigma$ can not be equal to $$\gamma=\left( \begin{array}{cc}
\zeta &0 \\ 
0 & 1
\end{array}\right)$$ where $\zeta$ is a $p$-primitive root of unity in $k_0.$ Indeed, the image of $\gamma$ via the homomorphism $\iota$ is the identity in $G/N.$ However, the image, by $\iota,$ of the element $$\tau=\left(\begin{array}{cc}
\zeta+\zeta^{-1}+1 &-1 \\ 
1 & 1
\end{array}\right)$$ which is a conjugate of $\gamma$ in $\mathrm{PGL}_2(k_0)$ is 
$$\overline{\tau}=\left( \begin{array}{cc}
3 &-1 \\ 
1 & 1
\end{array}\right).$$ Furthermore, for every odd prime $p,$ the 2 by 2 matrix $\overline{\tau}$ has order $p$ in $\mathrm{PGL}_2(\mathbb{F}_p).$ Note also that $\overline{\tau}$ and $\left( \begin{array}{cc}
1 &1 \\ 
0 & 1
\end{array}\right)$ are conjugate in $\mathrm{PGL}_2(\mathbb{F}_p).$
\end{exa}

\

With Theorem \ref{thm}, we are able to solve our lifting problem for certain type of cyclic curves. Indeed, it is clear that hyperelliptic curves are cyclic curves. The next theorem gives us all of possible finite groups of hyperelliptic curves over $k$  that can be lifted to characteristic $0.$
\begin{thm}\label{thm2}
Let $X/k$ be a smooth projective irreducible hyperelliptic curve. Denote by $G$ its full automorphism group and suppose that $p\neq2.$ Then, the automorphism group $G$ is liftable to characteristic $0$ if and only if, up to isomorphism, $G$ is one of the following groups:
\begin{enumerate}
\item $G=\mathbb{Z}/2p\mathbb{Z};$ 
\item $G=D_{2p};$
\end{enumerate}
with the exceptional possibilities:
\begin{enumerate}
\item $G=\mathbb{Z}/2\mathbb{Z}\times\mathcal{A}_5$ or $\mathrm{SL}_2(5)$ if $p=5;$ 
\item \label{tg} $G=\mathbb{Z}/2\mathbb{Z}\times\mathcal{A}_4,  \mathbb{Z}/2\mathbb{Z}\times\mathcal{S}_4, \mathbb{Z}/2\mathbb{Z}\times\mathcal{A}_5, \mathrm{SL}_2(3),$ or $\mathrm{GL}_2(3)$ in the case when $p=3.$ 
\end{enumerate}
In particular, the group $G$ is liftable to characteristic $0$ if and only if $G$ is an Oort-group for $k.$ We also observe that the order of $G$ is divisible by $p,$ but not by $p^2.$
\end{thm}
\begin{proof}
 Suppose that $G$ is liftable to characteristic $0.$ Denote by $\sigma$ the hyperelliptic involution of order $2$ of the group $G.$ Let $X_0/K_0$ be a good lifting of $X/k$ such that $G_0=\mathrm{Aut}_{k_0}(X_0)\simeq G.$ The curve $X_0$ is also a hyperelliptic curve (by Proposition \ref{pro}). The isomorphism between $G_0$ and $G$ induces an isomorphism between $G_0/{\langle\sigma\rangle}$ and $H=G/{\langle\sigma\rangle}.$ That is: Up to isomorphism, the group $H$ which is a finite subgroup of $\mathrm{PGL}_2(k)$ can be embedded in $\mathrm{PGL}_2(k_0).$ So, let us first recall the list of possible groups that can occur as full automorphism groups of the hyperelliptic curve $X_0/k_0.$ As far as we know, the list first appeared in \cite{listhyper}: 
 
 \
 
\begin{center} 
\begin{tabular}{|c|c|}
\hline 
$H$ & $G_0$\\
\hline
$\Z/n\Z$ & $\Z/2\Z\times\Z/n\Z, \Z/2n\Z$\\
\hline
$D_n$ & $\Z/2\Z\times D_n, V_n, D_{2n}, H_n, U_n, G_n$\\
\hline
$\mathcal{A}_4$ & $\Z/2\Z\times\mathcal{A}_4, \mathrm{SL}_2(3)$\\
\hline
 $\mathcal{S}_4$ & $\Z/2\Z\times\mathcal{S}_4, \mathrm{GL}_2(3), W_2, W_3$ \\
 \hline
 $\mathcal{A}_5$ & $\Z/2\Z\times\mathcal{A}_5, \mathrm{SL}_2(5)$ \\ 
\hline

\end{tabular}
 
 \
\end{center} 
where the groups $V_n, H_n, U_n, G_n, W_2$ and $W_3$ are defined as follows:
\begin{align*}
V_n&=\langle x,y \ {|} \ x^4, y^n,(xy)^2, (x^{-1}y)^2\rangle;\\
H_n&=\langle x,y \ {|} \ x^4, (xy)^n,x^2y^2\rangle;\\
U_n&=\langle x,y \ {|} \ x^2, y^{2n},xyxy^{n+1}\rangle;\\
G_n&=\langle x,y \ {|} \ x^2y^n,y^{2n}, x^{-1}yxy\rangle;\\
W_2&=\langle x,y \ {|} \ x^4, y^3,yx^2y^{-1}x^2, (xy)^4\rangle;\\
W_3&=\langle x,y \ {|} \ x^4, y^3,x^2(xy)^4, (xy)^8\rangle.\\
\end{align*}

Now, following Lemma \ref{biglem} and Theorem \ref{thm}, we distinguish $4$ cases:
\begin{itemize}
\item[$\underline{1^{\text{st}} case}:$] If $H\simeq \mathbb{Z}/p\mathbb{Z};$ 

That is $G_0/\langle\sigma\rangle\simeq \mathbb{Z}/p\mathbb{Z}.$ The abelian groups $\Z/2\Z\times\Z/p\Z$ and $\Z/2p\Z$ are isomorphic since $p$ is an odd prime. According to the list we have above, we conclude that  $G_0$ must be isomorphic to $\simeq\mathbb{Z}/2p\mathbb{Z}.$ 
\item[$\underline{2^{\text{nd}} case}:$] If $H\simeq D_{p};$

According to the list above, if $G_0/\langle\sigma\rangle\simeq D_n,$ for a given integer $n,$ then $G_0$ is isomorphic to one of the groups: $ \mathbb{Z}/2\mathbb{Z}\times D_n, D_{2n}, H_n, U_n V_n,$ and $G_n.$ However, in the cases when $G_0$ would be isomorphic to $V_n, H_n, U_n, G_n,$ or $\mathbb{Z}/2\mathbb{Z}\oplus D_n,$  the integer $n$ must be even (\cite{shaska1} Remark 6.). Since $p$ is assumed to be odd, then $G_0$ is isomorphic to $D_{2p}.$
%
\item[$\underline{3^{\text{rd}} case}:$] If $p\leq5$ and $H\simeq \mathcal{A}_5;$

Let $(\F_0{|}k_0,v)$ be the valued function field associated to $X_0/k_0$ where $v$ is the valuation such that $\F_0v=\F$ and $k_0=k.$ Let us consider the following polynomials:

\begin{align*}
R(x)&=x^{30}+522x^{25}-10005x^{20}-10005x^{15}-522x^5+1\\
S(x)&=x^{20}-228x^{15}+494x^{10}+228x^4+1\\
T(x)&=x^{10}+10x+1\\
G_i(x)&=(\lambda_i-1)x^{60}-36(19\lambda_i+29)x^{55}+6(26239\lambda_i-42079)x^{50}\\
&-540(23199\lambda_i-19343)x^{45}+105(737719\lambda_i-953143)x^{40}\\
&-72(1815127\lambda_i-145087)x^{35}-4(8302981\lambda_i+49913771)x^{30}\\
&+72(1815127\lambda_i-145087)x^{25}+105(737719\lambda_i-953143)x^{20}\\
&+540(23199\lambda_i-19343)x^{15}+6(26239\lambda_i-42079)x^{10}\\
&+36(19\lambda_i+29)x^{5}+(\lambda_i-1)\\
L&=\displaystyle{\prod_{i=1}^\delta}G_{i}.
\end{align*}
Let $x$ and $y$ be residually transcendental elements in $\F_0$ such that  $\F_0=k_0(x,y)$ and $y^2=F(x).$ According to \cite{shaska} $\S$ 4.5, we may assume that the polynomial $F$ is one of the following forms:

$$F=L, SL, TL, STL, RL, RSL, RTL, RSTL$$ where the $\lambda_i$'s, appearing in the $G_i$'s, are in $k_0$ and $\delta$ is the dimension of the Hurwitz space $\mathcal{H}(G_0,\mathbf{C}),$ the space of the family of covers
$$\varphi: \mathcal{X}_g\rightarrow\mathbb{P}^1$$ with fixed signature $\mathbf{C}$ and genus $g$ (the genus of $X_0$). We recall that the space $\mathcal{H}(G_0,\mathbf{C})$ is a finite dimensional subspace of the moduli space of genus $g$ hyperelliptic curves.

Note that, according to \cite{shaska} Table 1, if $F=L, SL, TL$ or $STL,$ the group $G_0$ is isomorphic to $\Z/2\Z\times\mathcal{A}_5.$ In the other cases, $G_0$ must be isomorphic to $\mathrm{SL}_2(5).$ 
%
%

For $p=3,$ the reduction modulo $3$ of the polynomials $G_{i}(x),R(x), S(x)$ are respectively $({\lambda}_i-1)(\overline{x}^{10}+1)^6, (\overline{x}^{10}+1)^3, (\overline{x}^{10}+1)^2$ and the polynomial $T$ is irreducible in characteristic $3.$ So, if we assume that $R$ or $ST$ divides the polynomial $F,$ then the genus of the residue function field $k(\overline{x},\overline{y})$ defined by $$\overline{y}^2=\overline{F}(\overline{x})$$ is $\overline{g}\geq1.$ But, note that $g>\overline{g}$ where $g$ denotes the genus of the function field $\F_0{|}k_0.$ Furthermore, we have
$$\left[k(\overline{x},\overline{y})\ {:} \ k(\overline{x})\right] =2$$ since the genus of $k(\overline{x},\overline{y})$ is non-zero. On the other hand, we have
$$\left[\F \ {:}\ k(\overline{x})\right]=2$$ and $\F\supseteq k(\overline{x},\overline{y}).$ Hence, $\F=k(\overline{x},\overline{y}).$ This is a contradiction. As, the valuation $v$ is a good reduction, we must have $$\overline{g}=g.$$ 
Therefore, the polynomials $ST$ and $R$ can not divide the polynomial $F.$ We conclude that,

$$F=L, SL, T.$$ In these cases, the residue function field $k(\overline{x},\overline{y})$ is of genus $0.$ That might be possible. So $G$ could be isomorphic to $\mathbb{Z}/2\mathbb{Z}\times\mathcal{A}_5.$

Now, if $p{=}5,$ the reduction modulo $5$ of the polynomials $G_{i}(x),$ $ R(x), S(x)$ and $T(x)$ are respectively $(\lambda_i-1)(\overline{x}^2+\overline{x}-1)^{30},$ $(\overline{x}+2)^5(\overline{x}-2)^{25}, (\overline{x}^2+\overline{x}-1)^{10}$ and $ (\overline{x}^2-1)^5.$ Using the same arguments as above, we must have,

$$F=L, SL, TL, STL, RL, RSG$$ and $G\simeq\mathbb{Z}/2\mathbb{Z}\times\mathcal{A}_5$ or $\mathrm{SL}_2(5).$
\item[$\underline{4^{\text{th}} case}:$] If $H\simeq \mathcal{A}_4$ or $\mathcal{S}_4;$

In both cases, we have $p=3.$ We shall use the same argument with the same notations as above.

First, suppose we have $H\simeq\mathcal{A}_4.$ We consider the following polynomials:

\begin{align*}
G_i&=x^{12}-\lambda_ix^{10}-33x^8+2\lambda_ix^6-33x^4-\lambda_ix^2+1\\
R&=x^4+2\sqrt{-3}x^2+1\\
S&=x^8+14x^4+1\\
T&=x(x^4-1)\\
L&=\displaystyle{\prod_{i=1}^\delta}G_{i}.
\end{align*}
According to T. Shaska in \cite{shaska}, we may assume that the function field $\F_0=k_0(x,y)$ is defined by
$$y^2=F(x)$$ with
$$F(x)=L, RL, SL, TL, TRL, TSL.$$
In the cases when, $F= L, RL, SL,$ the group $G_0$ is isomorphic to $\Z/2\Z\times\mathcal{A}_4.$ Otherwise, $G_0\simeq\mathrm{SL}_2(3).$

Among $R, S$ and $T,$ the polynomial $S$ is the only reducible polynomial modulo $3.$ And we have $S\equiv \ (x^4+1)^2 \ \mathrm{mod}\ 3.$ Hence, $S$ can not divide $F.$ Otherwise, it will contradict the fact that $v$ is a good reduction. So, the possible equations for $\F_0{|}k_0$ are the following:
$$y^2=L, RL, TL, TRL.$$ Which means the group $G_0$ could be isomorphic to one of the possibilities which are $\Z/2\Z\times\mathcal{A}_4$ and $\mathrm{SL}_2(3).$

Now suppose $H\simeq\mathcal{S}_4.$ In this case,  the polynomials $G_i, R, S$ and $T$ become:
\begin{align*}
G_i&=x^{24}+\lambda_ix^{20}+(759-4\lambda)x^{16}+2(3\lambda_i+1288)x^{12}\\
   & \ +(759-4\lambda)x^8+2\lambda_ix^6-33x^4+\lambda_ix^4+1\\
R&=x^{12}-33x^8-33x^4+1\\
S&=x^8+14x^4+1\\
T&=(x^4-1)\\
L&=\displaystyle{\prod_{i=1}^\delta}G_{i}.
\end{align*}
According to T. Shaska in \cite{shaska}, we may assume that the function field $\F_0=k_0(x,y)$ is defined by
$$y^2=F(x)$$ with
$$F(x)=L, SL, TL, STL, RL, RSL, RTL, RSTL.$$
In the cases when, $F= L$ or  $SL,$ the group $G_0$ is isomorphic to $\Z/2\Z\times\mathcal{S}_4.$ If $F=TL$ or $STL$, we have $G_0\simeq\mathrm{GL}_2(3).$ The group $G_0\simeq W_2$ in the cases when $F=RL$ or $RSL.$ And $G_0$ is isomorphic to $W_3$ for the rest of the possibilities.

Using the same reasoning again, the polynomials $R$ and $S$ are reducible modulo $3.$  We have $S\equiv \ (x^4+1)^2 \ \mathrm{mod}\ 3$ and $R\equiv \ (x^4+1)^3 \ \mathrm{mod}\ 3.$ But, $T$ is irreducible modulo $3.$ Which means $R$ and $S$ could not divide the polynomial $F.$ Hence, we must have:
$$F=L, TL.$$ We conclude that $G_0$ could be isomorphic to $\Z/2\Z\times\mathcal{S}_4$ or $\mathrm{GL}_2(3).$

\end{itemize}

For the converse, we shall use  Lemma \ref{rem2}. We know that any cyclic group is an Oort group. So, the groups $\mathbb{Z}/p\mathbb{Z}$ and $\mathbb{Z}/2p\mathbb{Z}$ are liftable to characteristic $0.$  

The dihedral group $D_{2p}$ is liftable to characteristic $0$ for $p\neq2$ since $D_{2p}\simeq\mathbb{Z}/2\mathbb{Z}\times D_p$ and $D_p$ with $\mathbb{Z}/p\mathbb{Z}$ are the only cyclic-by-$p$ subgroups which are Oort groups.
 
Finally, the groups  $\mathbb{Z}/2\mathbb{Z}\times\mathcal{A}_4,  \mathbb{Z}/2\mathbb{Z}\times\mathcal{S}_4, \mathbb{Z}/2\mathbb{Z}\times\mathcal{A}_5, \mathrm{SL}_2(3), \mathrm{GL}_2(3)$ and $\mathrm{SL}_2(5)$ are liftable to characteristic $0$ using the fact that their cyclic-by-$p$ subgroups are Oort groups for $p=3$ or $5.$  
 
\end{proof}

\

\begin{rem}
One natural question to ask is the following: Does all the groups on the list above (Theorem \ref{thm2}) occur as full automorphisms groups of hyperelliptic curves? Sanjeewa, in \cite{sanj}, gives a list of all finite groups that can occur as automorphisms groups of cyclic curves in any characteristic. All the groups on our list above appear in Sanjeewa's list in any prime characteristic, except, the group $\mathrm{SL}_2(5)$ in characteristic $5$ and the group $\mathrm{GL}_2(3)$ in characteristic $3.$ 

However, there exists a hyperelliptic curve in characteristic $3$ which has $\mathrm{GL}_2(3)$ as full automorphism group. The curve is defined as follows:
$$X/k: \ y^2=x^6+x^4+x^2+1 \ (\text{See} \ \cite{Ariyang}).$$
Moreover, according to Shaska in \cite{shaska} Example 5.2, in characteristic $0$, the curve defined by:
$$y^2=x^6+a_1x^4+a_2x^2+1$$ has $\mathrm{GL}_2(3)$ as full automorphism group if and only if $(u_1,u_2)=(-250,50)$ where
$$u_1=a_1^3+a_2^3, \ u_2=2a_1a_2.$$ Therefore, the curve defined in characteristic $0$ by 
$$X_0: \ y^2=x^6-5x^4-5x^2+1$$ has $\mathrm{GL}_2(3)$ as full automorphism group. Since $-5\equiv1 \ \mathrm{mod}\ 3,$ the automorphism group of $X/k$ is liftable to characteristic $0.$ Hence, the group $\mathrm{GL}_2(3)$ should appear in the list of Sanjaeewa in characteristic $3.$ 
\end{rem}

\

\begin{rem}
Theorem \ref{thm2} provides us with the list of the automorphism groups of hyperelliptic curves over $k,$ of characteristic $p\neq2,$ that can be lifted to characteristic $0.$ In the case when the characteristic of $k$ is equal to $2,$ our methods in the proof of the theorem do not seem to apply. The main reason is the fact that in characteristic $2,$ the minimal affine equation of the curve $X$ is given by 
$$y^2+P(x)y=F(x), \ P(x), F(x)\in k(x)$$ where $P(x)$ is possibly a non-zero polynomial.
\end{rem}

\

For the rest of the note, we assume that the prime characteristic of the base field $k$ is odd.

\

According to Shaska (See \cite{shaska1}), determining the automorphism group $G$ of a cyclic curve $X/k$ in the case when $p>2g+1$ is the same as in characteristic $0.$ Our next proposition gives a necessary and sufficient condition for $G$ to be liftable to characteristic $0$ in this case.

\begin{pro}
Let $X/k$ be a smooth cyclic curve of genus $g$ and denote by $G$ its full automorphism group. Assume that $p>2g+1.$ Let $n$ be a positive integer such that $(n,p)=1$ and $n$ is the order of the cyclic normal subgroup $N$ of $G$ such that the quotient space $X/N$ has genus $0.$ Then, the group $G$ is liftable to characteristic $0$ if and only if $p$ does not divide the order of $G.$
\end{pro}
\begin{proof}
Assume first that $X/k$ is a hyperelliptic curve. Let us prove the proposition by contradiction.

Suppose that $G$ is liftable to characteristic zero and $p$ divides the order of $G.$ We shall use the same notations we used in the proof of Theorem \ref{thm2}. 

By hypothesis, since the genus of the curve $X/k$ is always $\geq 2,$ we may assume that $p>5.$ According to Theorem \ref{thm2}, the group $G_0$ is isomorphic to $\mathbb{Z}/2p\mathbb{Z}$ or $D_{2p}.$ If $G_0\simeq\mathbb{Z}/2p\mathbb{Z},$ then the equation of the curve $X_0/k_0$ is one of the following (see \cite{shaska1}):
\begin{align*}
y^2 &=x^{2g+2}+a_1x^{p(t-1)}+\cdots +a_{\delta} x^p+1, t=\frac{2g+2}{p} \\
&=x^{2g+2}+a_1x^{p(t-1)}+\cdots +a_{\delta} x^p+1, t=\frac{2g+1}{p},  \text{or} \\
&=x(x^{pt}+a_1x^{p(t-1)}+\cdots +a_{\delta} x^p+1), t=\frac{2g}{p}.
\end{align*}

In all cases, we have $p\geq2g+2\geq pt.$ Since $2g+2$ is even, we conclude that $p>pt$ which is impossible. Thus, $G$ is not isomorphic to $\mathbb{Z}/2p\mathbb{Z}.$

Now, if $G\simeq D_{2p},$ then the equation of the curve $X_0/k_0$ is of the form: 
$$y^2=x.\prod_{{i=1}}^t(x^{2p}+\lambda_ix^p+1), \ t=\frac{g+1}{p}.$$
But, by hypothesis, we have $p\geq2g+2$ which contradicts the fact that $pt=g+1$ for some integer $t.$ 

We use the same reasoning in the case when the cyclic curve is not necessarily hyperelliptic. Indeed, with the assumption $(n,p)=1,$ the genus of the corresponding function field is
$$g=\frac{n-1}{2}(-1+pt) \ \text{or} \ (n-1)(-1+pt)$$ for some integer $t.$ This contradicts the hypothesis, $p>2g+1.$  
\end{proof}

\

Considering the results we have so far, we expect the following generalisation of Theorem \ref{thm2}:

\begin{conj}\label{conjc}
Let $X$ be a smooth cyclic curve over $k.$ Let $n$ be a positive integer such that $(2n,p)=1$ and $n$ is the order of the cyclic normal subgroup of $G=\mathrm{Aut}_{k}(X)$ such that the quotient space $X/N$ has genus $0.$ The group $G$ is liftable to characteristic $0$ if and only if $G$ is an Oort-group for $k.$
\end{conj}

Note that if Conjecture \ref{conjc} is true, using the results of Sanjeewa in \cite{sanj}, we have a complete list of all liftable automorphism groups of cyclic curves with the same hypothesis as in the conjecture. In the case when $p$ divides $n,$ we might have an automorphism group which is liftable to characteristic $0$ but not an Oort group for $k.$

\end{document}